\documentclass[11pt,a4 paper]{article}
\newtheorem{theo}{Theorem}[section]

\newtheorem{rem}{Remark}[section]

\newtheorem{lemma}{Lemma}[section]

\newtheorem{prop}{Proposition}[section]
\newtheorem{ex}{Example}[section]
\newtheorem{conj}{Conjecture}
\newtheorem{op}{Open Problem}

\usepackage{float}
\usepackage{tocloft}
\usepackage{amsfonts,amsmath,amssymb}
\usepackage{multirow, verbatim}
\usepackage{shadow,enumerate}
\usepackage{makeidx}  
\usepackage{graphicx}
\usepackage{color}
\def\Z{{\mathbb Z}}

\def\whitebox{{\hbox{\hskip 1pt
        \vrule height 6pt depth 1.5pt
        \lower 1.5pt\vbox to 7.5pt{\hrule width
                  3.2pt\vfill\hrule width 3.2pt}%
        \vrule height 6pt depth 1.5pt
        \hskip 1pt } }}
\def\qed{\ifhmode\allowbreak\else\nobreak\fi\hfill\quad\nobreak\whitebox\medbreak}

\begin{document}
\title{Generalized bent functions - sufficient \\ conditions and related constructions}

\author{
S. Hod\v zi\'c \footnote {University of Primorska, FAMNIT, Koper,
Slovenia, e-mail: samir.hodzic@famnit.upr.si}\and E.~Pasalic\footnote{
University of Primorska, FAMNIT \& IAM, Koper, Slovenia, e-mail: enes.pasalic6@gmail.com}
}

\date{}
\maketitle

\begin{abstract}
The necessary and sufficient conditions for a class of functions $f:\mathbb{Z}_2^n \rightarrow \mathbb{Z}_q$, where $q \geq 2$ is an even positive integer, have been recently identified for $q=4$ and $q=8$. In this article we give an alternative characterization of the generalized Walsh-Hadamard transform in terms of the Walsh spectra of the component Boolean functions of $f$, which then allows us to derive sufficient conditions that $f$ is generalized bent for any even $q$. The case when $q$ is not a power of two, which has not been addressed previously,  is treated separately and a suitable representation in terms of the component functions is employed. Consequently, the derived results lead to generic construction methods of this class of functions. The main remaining task, which is not answered in this article, is whether the sufficient conditions are also necessary. There are some indications that this might be true which is also formally confirmed for generalized bent functions that belong to the class of generalized Maiorana-McFarland functions (GMMF), but still we were unable to completely specify (in terms of necessity)  gbent conditions.


\end{abstract}

\section{Introduction}
A generalization  of Boolean  functions    was introduced  in \cite{Kum85} for considering a much larger class of mappings from $\mathbb{Z}_q^n$ to $\mathbb{Z}_q$. Nevertheless, due to a more natural connection to cyclic codes over rings, functions from $\mathbb{Z}_2^n$ to $\mathbb{Z}_q$, where $q \geq 2$ is a positive integer, have drawn even more attention \cite{KUSchm2007}.  In \cite{KUSchm2007}, Schmidt studied the relations between
generalized bent functions, constant amplitude codes and $\mathbb{Z}_4$-linear codes ($q=4$).
 The latter class of mappings is called {\em generalized bent (gbent) functions} throughout this article. There are also other generalizations of bent functions such as bent functions over finite Abelian groups for instance \cite{Solod2002}. A nice survey on different generalizations of bent functions can be found in \cite{Tokareva}.

 There are several reasons for studying generalized bent functions. In the first place there is a close connection of these objects to standard bent functions and for instance the relationship between the bent conditions imposed on the component functions of gbent functions (using a suitable decomposition) for the quaternary and octal case were investigated in \cite{Tok} and \cite{Octal}, respectively.
 Also, in many other recent works \cite{SecCon, OnCross, BentGbent} the authors mainly consider the bentness of the component functions for a given prescribed form of a gbent functions. In particular, it was shown in \cite{BentGbent} that some standard classes of bent functions such as Mariaona-McFarland class and Dillon's class naturally induce gbent functions.  A particular class of the functions represented as $f(x)=c_1a(x)+c_2b(x)$ were thoroughly investigated in terms of the imposed conditions on the coefficients $c_i \in \mathbb{Z}_q$ and the choice of the Boolean functions $a$ and $b$, so that $f$ is gbent \cite{SH}. A more interesting research challenge in this context is to propose some direct construction methods of functions  from $\mathbb{Z}_2^n$ to $\mathbb{Z}_q$, which for suitable $q$ may give a nontrivial decomposition into standard bent functions that possibly do not belong to the known classes of bent functions. The second reason for the interest in these objects is a close relationship between certain objects used in the design of orthogonal frequency-division multiplexing (OFDM)
modulation technique which in certain cases suffers from relatively high peak-to-mean envelope power ratio
(PMEPR). To overcome these issues, the $q$-ary sequences lying in complementary pairs \cite{Golay} (also called Golay sequences) having a low PMEPR  can be
easily determined from the generalized Boolean function associated with this sequence, see \cite{KUGenRM} and the references therein.

\bigskip
In this article, we address an important problem of specifying the conditions that $f:\mathbb{Z}_{2}^n \rightarrow \mathbb{Z}_q$ is gbent. In difference to the previous work \cite{BentGbent,Octal}, where the sufficient and necessary conditions when $q=4$ and $q=8$ were derived, we consider the general case of $q$ being even and subsequently derive some sufficient conditions for $f$ to be gbent. We emphasize the fact that the sufficient and necessary conditions for $q=8$ were derived in a nontrivial manner employing so-called Jacobi sums and the same technique could not be applied for larger $q$ of the form $2^h$. Nevertheless, our sufficient conditions completely coincide in this case and therefore they are also necessary as well. That our sufficient conditions may also be necessary at the same time is further supported by the fact that the GMMF class of gbent functions essentially satisfies these conditions, see Section~\ref{sec:GMMF}. The major difficulty in proving that the sufficiency is at the same necessity as well lies is the hardness of dealing with certain character sums.

The whole approach and the sufficient conditions derived here is based on an alternative characterization and computation of the generalized Walsh-Hadamard spectral values through using the standard Walsh spectra of the component Boolean functions $a_i$ when $f:\mathbb{Z}_{2}^n \rightarrow \mathbb{Z}_q$  is (uniquely) represented as $f(x)=a_0(x) + 2a_1(x)+ \cdots + 2^{h-1}a_{h-1}(x)$. While this representation allows for a relatively easy treatment of the case $q=2^h$, it turns out that it is not so efficient when considering even $q$ in the range $2^{h-1} < q < 2^{h}$. Even though given the input and output values this representation is still unique for even $2^{h-1} < q < 2^{h}$, to give some sufficient conditions for the gbent property in this case we were forced to consider a different form of $f$ which necessarily contains the coefficient $q/2$ in its representation. Thus, in this case (again to avoid some difficult character sums) the function $f$ is rather represented as $f(x)=\frac{q}{2} a(x)+ a_0(x) + 2a_1(x)+ \cdots + 2^{h-2}a_{h-2}(x)$ which then simplify the analysis of their properties. Using these representations we derive a compact and simple formula to compute the generalized Walsh-Hadamard spectra in terms of the spectra of the component functions of $f$. Based on this formula some sufficient conditions for the gbent property are derived which in turn gives us the possibility to specify certain generic classes of gbent functions.


The rest of this article is organized as follows. In Section~\ref{sec:prel}, some basic definitions concerning (generalized) bent functions are given. A new convenient formula for computing the generalized Walsh-Hadamard spectra of $f:\mathbb{Z}^n_2\rightarrow \mathbb{Z}_q$ in terms of the spectral values of its component functions is derived in Section~\ref{sec:GWHT}. In Section~\ref{sec:suffic}, a set of sufficient conditions on the Walsh spectra of the Boolean component functions of $f$, ensuring that $f$ is gbent, are specified. It turns out that in some particular cases these conditions are also necessary, but whether these sufficient conditions are also necessary, in general, is left as an open problem. The problem of designing gbent functions, satisfying the set of sufficient conditions introduced previously, is addressed in Section~\ref{sec:conditions} one trivial method for this purpose are given. The task of selecting the component functions, that satisfy the set of sufficient conditions, in a non-trivial way appears to be rather difficult.
 Some concluding remarks are given in  Section~\ref{sec:concl}.

\section{Preliminaries}\label{sec:prel}
We denote the set of integers, real numbers and complex numbers by $\mathbb{Z},$ $\mathbb{R}$ and $\mathbb{C}$, respectively, and  the ring of integers modulo $r$ is denoted by $\mathbb{Z}_r.$ The vector space $\mathbb{Z}_2^n$ is the space of all $n$-tuples $x=(x_1,\ldots,x_n)$, where $x_i \in \mathbb{Z}_2$ with the standard operations.  For $x=(x_1,\ldots,x_n)$ and $y=(y_1,\ldots,y_n)$  in $\mathbb{Z}^n_2$, the scalar (or inner) product over $\mathbb{Z}_2$ is defined as $x\cdot y=x_1 y_1\oplus\cdots\oplus x_n y_n.$ The same inner product of two vectors $x, y \in \mathbb{Z}_q$, when defined modulo $q$, will be denoted by ``$\odot$'', thus  $x\odot y=x_1 y_1+ \cdots + x_n y_n \pmod{q}.$ The addition over $\mathbb{Z},$ $\mathbb{R}$ and $\mathbb{C}$ is denoted by ``+", but also the addition modulo $q$  and it should be understood from the context when reduction modulo $q$ is performed.  The  binary addition  over $\mathbb{Z}_2$ is denoted by $\oplus$ in a few cases when we use this addition. The cardinality  of the set $S$ is denoted by $|S|$. If $z=u+vi\in \mathbb{C},$ then $|z|=\sqrt{u^2+v^2}$ denotes the absolute value of $z,$ and $\overline{z}=u-vi$ denotes the complex conjugate of $z,$ where $i^2=-1,$ and $u,v\in \mathbb{R}.$ We also denote $u=\mathfrak{Re}(z)$ and $v=\mathfrak{Im}(z).$

The set of all Boolean functions in $n$ variables, that is the mappings from $\mathbb{Z}_2^n$ to $\mathbb{Z}_2$ is denoted by $\mathcal{B}_n$. Especially, the set of affine functions in $n$ variables we define as $\mathcal{A}_n=\{g(x)=a\cdot x\oplus b\;|\;a\in\mathbb{Z}_2^n,\; b\in\{0,1\}\}.$
The \emph{Walsh-Hadamard transform} (WHT) of $f\in\mathcal{B}_n$ at any point $\omega\in\mathbb{Z}^n_2$ is defined by $$W_{f}(\omega)=2^{-\frac{n}{2}}\sum_{x\in \mathbb{Z}^n_2}(-1)^{f(x)\oplus \omega\cdot x}.$$ A function $f\in\mathcal{B}_n,$ where $n$ is even, is called {\em bent} if and only if $|W_f(\omega)|=1$ for all $\omega\in\mathbb{Z}^n_2$, and these functions only exist  for $n$ even. If $n$ is odd, a function $f\in\mathcal{B}_n$ is said to be semibent if $W_f(\omega)\in\{0,\pm \sqrt{2}\}$, for every $\omega\in \mathbb{Z}_2.$ We call a function from $\mathbb{Z}^n_2$ to $\mathbb{Z}_q$ ($q\geq 2$ a positive integer) a \emph{generalized Boolean function} in $n$ variables \cite{Tok}. We denote the set of such functions by $\mathcal{GB}^n_q$ and for  $q=2$ the classical Boolean functions in $n$ variables are obtained.

Let $\zeta=e^{2\pi i/q}$ be a complex $q$-primitive root of unity. The {\em generalized Walsh-Hadamard transform} (GWHT) of $f\in \mathcal{GB}^q_n$ at any point $\omega\in \mathbb{Z}^n_2$ is the complex valued function $$\mathcal{H}_{f}(\omega)=2^{-\frac{n}{2}}\sum_{x\in \mathbb{Z}^n_2}\zeta^{f(x)}(-1)^{\omega\cdot x}.$$
A function $f\in \mathcal{GB}^q_n$ is called \emph{generalized bent (gbent)} function if $|\mathcal{H}_f(\omega)|=1$, for all $\omega\in \mathbb{Z}^n_2.$  If $q=2$, we obtain the (normalized) Walsh transform $W_f$ of $f\in\mathcal{B}_n$. Two n-variable Boolean functions $f,g\in \mathcal{B}_n$ are said to
be disjoint spectra functions if $W_f(\omega)W_g(\omega)=0,$ for every $\omega\in \mathbb{Z}^n_2$ \cite{CCA}.

A $(1,-1)$-matrix $H$ of order $p$ is called a \emph{Hadamard} matrix if  $HH^{T}=pI_p,$ where $H^{T}$ is the transpose of $H$, and $I_p$ is the $p\times p$ identity matrix. A special kind of  Hadamard matrix is the \emph{Sylvester-Hadamard} or \emph{Walsh-Hadamard} matrix, denoted by $H_{2^{k}},$ which is constructed using Kronecker product $H_{2^{k}}=H_2\otimes H_{2^{k-1}},$ where
\begin{eqnarray*}\label{HM}
H_1=(1);\hskip 0.4cm H_2=\left(
                           \begin{array}{cc}
                             1 & 1 \\
                             1 & -1 \\
                           \end{array}
                         \right);\hskip 0.4cm H_{2^k}=\left(
      \begin{array}{cc}
        H_{2^{k-1}} & H_{2^{k-1}} \\
        H_{2^{k-1}} & -H_{2^{k-1}} \\
      \end{array}
    \right).
\end{eqnarray*}
For a function $g$  on $\mathbb{Z}^n_2,$ the $(1,-1)$-sequence defined by $((-1)^{g(v_0)},(-1)^{g(v_1)},\ldots,(-1)^{g(v_{2^{n}-1})})$ is called the \emph{sequence} of $g$, where $v_i=(v_{i,0},\ldots,v_{i,n-1}),$ $i=0,1,\ldots,2^{n}-1,$ denotes the vector in $\mathbb{Z}^n_2$ whose integer representation is $i$, that is, $i=\sum_{j=0}^{n-1}v_{i,j} 2^j$.
We take that $\mathbb{Z}^n_2$ is ordered as $$\{(0,0,\ldots,0),(1,0,\ldots,0),,(0,1,\ldots,0),\ldots,(1,1,\ldots,1)\},$$
and the vector $v_i=(v_{i,0},\ldots,v_{i,n-1}) \in \mathbb{Z}^n_2$ is uniquely identified by $i \in \{0,1,\ldots,2^n-1\}$.


\section{Motivation and Conjecture on GWHT}\label{sec:GWHT}

In this section, we recall some results related to quaternary and octal gbent functions \cite{Tok,Octal} in  terms of GWHT. The necessary and sufficient conditions for gbent property derived in \cite{Tok,Octal} for $q=4$ and $q=8$ motivates us to conjecture that similar sufficient conditions are valid for arbitrary even $q$, which is then proved in Section~\ref{sec:conditions}. Notice that proving the necessity of these conditions  turns out   to be hard though there are certain indications that the sufficient conditions given in Theorem~\ref{mainth2} are also necessary.

If $2^{h-1}<q\leq 2^h,$ to any generalized function $f:\mathbb{Z}^n_2\rightarrow \mathbb{Z}_q,$ we may associate a unique sequence of Boolean functions $a_i\in \mathcal{B}_n$ ($i=0,1,\ldots,h-1$) such that
\begin{eqnarray}\label{eq:1}
f(x)=a_0(x)+2a_1(x)+2^2a_2(x)+\ldots+2^{h-1}a_{h-1}(x),\; \forall x\in \mathbb{Z}^n_{2}.
\end{eqnarray}
The functions $a_i(x),$ $i=0,1,\ldots,h-1,$ are called the component functions of the function $f(x).$ When $q=4$ it was shown that the function $f(x)=a_0(x)+2a_1(x),$ $a_0,a_1\in \mathcal{B}_n,$ is gbent if and only if $a_1(x)$ and $a_0(x)\oplus a_1(x)$ are bent Boolean functions \cite{Tok}. Note that the last condition implies  that $a_0(x)$ is not necessarily bent (it can be affine for instance), and consequently  only $a_1(x)$ needs to be bent. In addition, the GWHT of the function $f$ in this case is expressed in  terms of the WHT transforms of the functions $a_1(x)$ and $a_0(x)\oplus a_1(x),$ i.e., we have
\begin{eqnarray*}
\mathcal{H}_f(u)=\frac{1}{2}[(W_{a_1}(u)+W_{a_0\oplus a_1}(u))+i(W_{a_1}(u)-W_{a_0\oplus a_1}(u))],\; \forall u\in \mathbb{Z}^n_{2}.
\end{eqnarray*}
However, we may rewrite this equality so that we view $\mathcal{H}_f$ as a linear combination of $W_{a_1}$ and $W_{a_0\oplus a_1},$ where the coefficients are complex numbers, that is,
\begin{eqnarray}\label{eq:2}
\mathcal{H}_f(u)=\left(\frac{1}{2}+i\right)W_{a_1}(u)+\left(\frac{1}{2}-i\right)W_{a_0\oplus a_1}(u).
\end{eqnarray}

In the case when $q=8$, for $f\in\mathcal{GB}^8_n$ given by
\begin{eqnarray}\label{eq:3}
f(x)=a_0(x)+2a_1(x)+2^2a_2(x),
\end{eqnarray}
the GWHT of $f$ is given by the following lemma.
\begin{lemma}\cite{BentGbent,Octal}\label{lemma1}
Let $f\in \mathcal{GB}^8_n$ as in (\ref{eq:3}). Then,
\begin{eqnarray}\label{eq:4}
4\mathcal{H}_f(u)=\alpha_0 W_{a_2}(u)+\alpha_1 W_{a_0\oplus a_2}(u)+\alpha_2 W_{a_1\oplus a_2}(u)+\alpha_3 W_{a_0\oplus a_1\oplus a_2}(u),
\end{eqnarray}
where $\alpha_0=1+(1+\sqrt{2})i$, $\alpha_1=1+(1-\sqrt{2})i,$ $\alpha_2=1+\sqrt{2}-i,$ $\alpha_3=1-\sqrt{2}-i.$
\end{lemma}
\begin{rem}\label{rem:triv}
A special case of selecting $a_0(x)=0$ appears to be interesting. In the first place, the condition relating the Walsh coefficients becomes simpler, that is, $$4\mathcal{H}_f(u)=2(1+i) W_{a_2}(u)+2(1-i) W_{a_1\oplus a_2}(u), \; \forall u\in \mathbb{Z}^n_{2}.$$
Then, assuming further that $a_1(x)=0$ would actually give $4\mathcal{H}_f(u)=4W_{a_2}(u)$, meaning that we only have one bent  function and that the function $f(x)=4a_2(x)$ is gbent though its codomain only takes the values from the set $\{0,4\}$. In general, any function defined as $f(x)=\frac{q}{2}a(x)$ is gbent if and only if $a(x)$ is a bent function.
\end{rem}
\begin{rem} Apart form the trivial case discussed in Remark \ref{rem:triv}, we may also consider other suitable choices for the component functions $a_{0}, a_{1}$ and $a_{2}$. Fixing $a_{2}$ to be bent we may consider  $a_0, a_{1} \in \mathcal{A}_n$ to be suitably chosen affine functions so that the above conditions are satisfied. Indeed, since $a_{2}$ being bent implies that the addition of any affine function to it does not affect the bent property we can assume that $a_i \in \mathcal{A}_n$ for $i=0,1$.
It is well-known that for $a_i(x)=a_{i,0} +a_{i,1}x_1 + \ldots + a_{i,n}x_n$, if the Walsh transform of $f(x)$ at point $u$ is $W_f(u)$ then the transform of $f(x)+a_i(x)$ at point $u$ is $(-1)^{a_{i,0}}W_f(u + a^{(i)})$, where $a^{(i)} \in \Z_2^n$ is given as $a^{(i)}=(a_{i,1}, \ldots, a_{i,n})$. Hence, (\ref{eq:4})
can be rewritten as,
$$4\mathcal{H}_f(u)=\alpha_0 W_{a_2}(u)+\alpha_1 (-1)^{a_{0,0}}W_{a_2}(u + a^{(0)})+\alpha_2 W_{a_1\oplus a_2}(u)+\alpha_3 (-1)^{a_{0,0}}W_{ a_1\oplus a_2}(u + a^{(0)}).$$

\end{rem}
Notice that $\mathcal{H}_f$ in (\ref{eq:4}) is again a linear combination of the WHTs of the functions $a_2(x),$ $a_0(x)\oplus a_2(x)$, $a_1(x)\oplus a_2(x)$, $a_0(x)\oplus a_1(x)\oplus a_2(x)$.
Moreover, the following theorem imposes the conditions for the function $f \in \mathcal{GB}^8_n$ to be a gbent function.
\begin{theo}\cite{BentGbent}\label{theo1}
Let $f\in \mathcal{GB}^8_n$ as in (\ref{eq:3}). Then:
\begin{enumerate}[1)]\label{th1}
\item If $n$ is even, then $f$ is generalized bent if and only if $a_2,$ $a_0\oplus a_2$, $a_1\oplus a_2$, $a_0\oplus a_1\oplus a_2$ are all bent, and
$$(*)\; W_{a_0\oplus a_2}(u)W_{a_1\oplus a_2}(u)=W_{a_2}(u)W_{a_0\oplus a_1\oplus a_2}(u), \;\; \textnormal{for all } u\in \mathbb{Z}^n_2;$$
\item If $n$ is odd, then $f$ is generalized bent if and only if $a_2,$ $a_0\oplus a_2$, $a_1\oplus a_2$, $a_0\oplus a_1\oplus a_2$ are semi-bent satisfying $$(**): W_{a_0\oplus a_2}(u)=W_{a_2}(u)=0 \;\;\; \wedge \;\;\;|W_{a_1\oplus a_2}(u)|=|W_{a_0\oplus a_1 \oplus a_2}(u)|=\sqrt{2}; \; \textnormal{ or}$$ $$W_{a_1\oplus a_2}(u)=W_{a_0\oplus a_1\oplus a_2}(u)=0 \;\;\; \wedge \; \;\; |W_{a_0\oplus a_2}(u)|=|W_{a_2}(u)|=\sqrt{2},$$
for all $u\in \mathbb{Z}^n_2.$
\end{enumerate}
\end{theo}
In general, a formula which gives the GWHT of the function $f$ given by (\ref{eq:1}) is given by the following theorem.
\begin{theo}\cite{BentGbent,Octal}
The Walsh-Hadamard transform of $f:\mathbb{Z}^n_2\rightarrow \mathbb{Z}_q$, $2^{h-1}<q\leq 2^h,$ where $f(x)=\sum^{h-1}_{i=0}a_i(x)2^i,$ $a_i\in\mathcal{B}_n$ is given by
\begin{eqnarray}\label{eq:6}
\mathcal{H}_f(u)=2^{-h}\sum_{I\subseteq \{0,\ldots,h-1\}}\zeta^{\sum_{i\in I}2^i}\sum_{J\subseteq I, K\subseteq \overline{I}}(-1)^{|J|}W_{\sum_{t\in J\cup K}a_{t}(x)}(u).
\end{eqnarray}
\end{theo}
This implicit expression does not reveal the fact that $\mathcal{H}_f$ of a function $f$ represented as in  (\ref{eq:1}) can be given explicitly as a linear combination (with complex coefficients that can be efficiently computed) of the WHTs of some linear combinations of its component functions $a_i(x)$, $i=0,1,\ldots, h-1.$
Therefore, for an arbitrary generalized Boolean function $f$ given by (\ref{eq:1}), it is of great importance to develop a more useful formula for its GWHT which will be given in the next section.

 
Before we state our conjecture regarding the GWHT and the conditions ($*$)-($**$) in general, we first formalize our observations. Let $\Theta_i(x)$ be the function defined as
\begin{eqnarray}\label{theta-sec1.1}
\Theta_{i}(x)=(-1)^{z_{i,0} a_0(x)\oplus z_{i,1} a_1(x)\oplus \ldots\oplus z_{i,h-1}a_{h-1}(x)},
\end{eqnarray}
where $z_i=(z_{i,0},z_{i,1},\ldots,z_{i,h-1})\in \mathbb{Z}^{h}_2$  and $i$ denotes its integer representation, $i=0,\ldots,2^h-1$.
\begin{rem}\label{rem-sec1}
Note that the  function $\Theta_{i}(x)$ actually gives $(-1)$ powered to all possible linear combinations of  the component functions $a_0(x),a_1(x),\ldots,a_{h-1}(x)$. In addition, we always have $\zeta^{\frac{q}{2}a_{h-1}(x)}=(-1)^{a_{h-1}(x)}$ for $q=2^h$.
\end{rem}
For  $q=8=2^3$, thus $h=3$, let us consider $f:\mathbb{Z}^n_2\rightarrow \mathbb{Z}_8$  given by (\ref{eq:3}). Since $\zeta^{4a_2(x)}=(-1)^{a_2(x)},$ the GWHT is given as:
\begin{eqnarray}\label{eq:7}
\mathcal{H}_f(u)=\sum_{x\in \mathbb{Z}^n_2}\zeta^{f(x)}(-1)^{u\cdot x}=\sum_{x\in \mathbb{Z}^n_2}\zeta^{a_0(x)+2a_1(x)}(-1)^{a_{2}(x)\oplus u\cdot x}.
\end{eqnarray}
Hence, for $q=8$ we have $z=(z_0,z_1)\in\mathbb{Z}^{2}_2$, $\Theta_{z}(x)=(-1)^{z_0 a_0(x)\oplus z_1 a_1(x)}$, where
\begin{eqnarray}\label{theta8} \nonumber
\Theta_{0}(x)=\Theta_{(0,0)}(x)&=&1,  \\ \nonumber
\Theta_{1}(x)=\Theta_{(1,0)}(x)&=&(-1)^{a_0(x)}, \\
\Theta_{2}(x)=\Theta_{(0,1)}(x)&=&(-1)^{a_1(x)}, \\ \nonumber
\Theta_{3}(x)=\Theta_{(1,1)}(x)&=&(-1)^{a_0\oplus a_1(x)},   \nonumber
\end{eqnarray}
and
\begin{eqnarray*}
\zeta^{a_0(x)+2a_1(x)}=2^{-2}(\alpha_0\Theta_{0}(x)+\alpha_1\Theta_{1}(x)+\alpha_2\Theta_{2}(x)+\alpha_3\Theta_{3}(x)),
\end{eqnarray*}
where $\alpha_i$ are given in Lemma \ref{lemma1}.
Regarding the GWHT of the function $f$ given by (\ref{eq:1}), we propose the following conjecture.
\begin{conj}\label{conj}
Let $f\in \mathcal{GB}^q_n$ and $\Theta_i(x)$ be given by (\ref{eq:1}) and (\ref{theta-sec1.1}), respectively. Then $\zeta^{f(x)}$ can be represented as a linear combination of functions $\Theta_{i}(x),$ $i=0,1,\ldots,2^{h}-1,$ where the coefficients $\alpha_i$ are complex numbers, i.e.,
\begin{eqnarray}\label{conj-1}
\zeta^{f(x)}=\zeta^{\sum^{h-1}_{i=0}a_i(x)2^i}=\sum^{2^h-1}_{i=0}\alpha_i \Theta_{i}(x).
\end{eqnarray}
Furthermore, for a given $f \in \mathcal{GB}^q_n$ the coefficients $\alpha_i$ can be computed efficiently.
\end{conj}
Note that Conjecture \ref{conj} covers all the  values of even $q$ in the range  $q \in (2^{h-1}, 2^{h}]$. Clearly, in the case when $q=8=2^h$ (similarly when $q=4$), we had that $\zeta^{a_0(x)+2a_1(x)+4a_2(x)}=(-1)^{a_2(x)}\zeta^{a_0(x)+2a_1(x)},$ and consequently we represent only $\zeta^{a_0(x)+2a_1(x)}$ as a linear combination of functions $\Theta_{i}(x),$ $i=0,1,2,3.$
This representation is proved useful later for deriving sufficient conditions of gbent property and for  generalizing  Theorem \ref{theo1} but covering all values of $q$, where $q$ is even.

\subsection{New GWHT formula}\label{sec:proof1}

In this section, we prove Conjecture \ref{conj} and consequently  a new GWHT formula for any generalized function $f\in \mathcal{GB}^q_n$, which computes $\mathcal{H}_f$ by using the Walsh spectral values of the component functions and the coefficients $\alpha_i$, is derived.

Let $f:\mathbb{Z}^n_2\rightarrow \mathbb{Z}_q$, $2^{h-1}<q\leq 2^h,$ where again $f(x)=a_0(x)+2a_1(x)+\ldots+
2 ^{h-1}a_{h-1}(x)$, $a_i(x) \in \mathcal{B}_n$. For convenience, we introduce the coefficients $c_i=2^{i}$, for $i=0,\ldots,h-1$, thus writing $f(x) =\sum_{i=0}^{h-1}c_ia_i(x)$. Notice that whatever formal representation of $f$ is used (see also Example~\ref{ex:alphai}), once the function $f$ has been specified in terms of its input and output values, the decomposition into the Boolean component function $a_i(x)$ as given above is unique and any other representation can be transformed into this form.


Assume now that the function $f$ can be represented as a linear combination of the functions $\Theta_{i}(x)$ as in (\ref{conj-1}), that is,
\begin{eqnarray}\label{p1}
\zeta^{f(x)}=\zeta^{\sum^{h-1}_{i=0}c_i a_i(x)}=\sum^{2^h-1}_{i=0}\alpha_i \Theta_{i}(x),
\end{eqnarray}
for some complex numbers $\alpha_i\in \mathbb{C}$ and $\Theta_i(x)=(-1)^{z_{i,0} a_0(x)\oplus \cdots\oplus z_{i,h-1}a_{h-1}(x)},$ as given by (\ref{theta-sec1.1}).
 The main task is to find the coefficients $\alpha_i$ such that (\ref{p1}) holds for every $x\in \mathbb{Z}^n_2.$

Consider an arbitrary but fixed $x'\in \mathbb{Z}^n_2$ such that $(a_0(x'),\ldots,a_{h-1}(x'))=z_k\in\mathbb{Z}^h_2,$ where $k$ is the integer representation of a binary vector $z_k.$ To relate the functions $\Theta_i$ to the rows (columns) of the  Hadamard matrix we need the following useful identification. It is well-known that the rows of the Hadamard matrix $H_{2^h}$ of size $2^h \times 2^h$ are the evaluations of all  linear functions in $\mathcal{B}_h$, that is, the $k$-th row of $H_{2^h}$ (alternatively the $k$-th column since $H_{2^h}=H^T_{2^h}$) can be expressed as $H_{2^h}^{(k)}=\{(-1)^{z_k \cdot y} \mid y \in \mathbb{Z}_2^h \}$, where $z_k$ is fixed. Therefore,

\begin{eqnarray*}\label{seq}
(\Theta_{0}(x'),\Theta_{1}(x'),\ldots,\Theta_{2^h-1}(x'))=H^{(k)}_{2^h}.
\end{eqnarray*}
Indeed, for a fixed $x' \in \mathbb{Z}_2^n$ the value of a binary vector $(a_0(x'),\ldots,a_{h-1}(x'))=z_k$ is also fixed and it is easy to verify that,
$$(\Theta_{0}(x'),\Theta_{1}(x'),\ldots,\Theta_{2^h-1}(x'))=((-1)^{z_k\cdot z_0},(-1)^{z_k\cdot z_1},\ldots,(-1)^{z_k\cdot z_{2^h-1}})=H^{(k)}_{2^h},$$
where $z_0,z_1,\ldots,z_{2^h-1}$ are elements of the set $\mathbb{Z}^n_2.$ 
Furthermore, for this particular (but arbitrary) value $x'$ the fact that $(a_0(x'),\ldots,a_{h-1}(x'))=z_k$ implies that
\begin{eqnarray}\label{exponent}
\zeta^{f(x')}=\zeta^{\sum^{h-1}_{i=0}c_i a_{i}(x')}=\zeta^{z_k\odot (c_0,\ldots,c_{h-1})}.
\end{eqnarray}
Now, if we define the column matrix $\Lambda=[\alpha_i]^{2^h-1}_{i=0}$ to be a matrix of the coefficients $\alpha_i,$ the previous discussion together with (\ref{p1}) implies that
$$H^{(k)}_{2^h}\left(
    \begin{array}{c}
      \alpha_0 \\
       \alpha_1 \\
      \vdots \\
       \alpha_{2^{h}-1} \\
    \end{array}
  \right)_{2^h \times 1}=H^{(k)}_{2^h}\Lambda=\zeta^{z_k\odot (c_0,\ldots,c_{h-1})}.$$
Notice that when $z_k$ goes through $\mathbb{Z}^h_2$ the value $z_k \odot (c_0,\ldots,c_{h-1})$ goes through $\mathbb{Z}_{q}$, since the operation $\odot$ means cutting by modulo $q$. Therefore,
 it is convenient to  define a column matrix $B$ as a matrix of all corresponding powers of $\zeta$, that is, $B=[\zeta^{z_i\odot (c_0,\ldots,c_{h-1})}]^{2^h-1}_{i=0}$ or given in the matrix form as,
\begin{eqnarray}\label{BX}
   B=\left(
                     \begin{array}{c}
                       \zeta^0 \\
                       \zeta^{c_0} \\
                       \vdots \\
                       \zeta^{c_0+\cdots+c_{2^h-1}} \\
                     \end{array}
                   \right).
\end{eqnarray}

  and obviously assuming (\ref{p1}) is valid the following system of equations must be satisfied
\begin{eqnarray}\label{sis}
H_{2^h}\Lambda=B.
\end{eqnarray}

As mentioned previously, the function $f \in  \mathcal{GB}^q_n$  may be given  in  different forms, for instance $f(x)=\sum_{i=0}^d c_ib_i(x)$, where $b_i \in \mathcal{B}_n$ but $c_i \in \mathbb{Z}_q$ and in general $c_i \neq 2^i$. Nevertheless, one can easily transform such a function into the form discussed above.
Note that the solution $\Lambda$ of the system (\ref{sis}) implies that the equality (\ref{p1}) holds for any $x\in \mathbb{Z}^n_2.$ The main reason for this is the fact that the Hadamard matrix covers all possible values of the vector $(\Theta_0(x),\Theta_1(x),\ldots,\Theta_{2^h-1}(x))$.
Therefore, for any $x \in \mathbb{Z}^n_2$ the evaluation  of the component functions $(a_0(x),\ldots,a_{h-1}(x))$ implies that the corresponding Hadamard row multiplied with  $\Lambda$ will always be equal to the corresponding power of $\zeta.$

Since the determinant of the Sylvester-Hadamard matrix is given as $\det(H_{2^h})=\pm 2^{h2^{h-1}}$, using the fact that $H^{-1}_{2^h}=2^{-h}H^{T}_{2^h}$ ($H_{2^h}$ is symmetric), we have that the unknown column matrix $\Lambda=[\alpha_i]^{2^{h}-1}_{i=0}$  is (uniquely) given by
\begin{eqnarray}\label{X}
\Lambda=H_{2^h}^{-1}B=2^{-h}H^{T}_{2^h}B=2^{-h}H_{2^h}B.
\end{eqnarray}
In the following example, we illustrate a complete procedure of finding $\alpha_i$ with respect to both discussed representations of the function $f(x)$.
\begin{ex}\label{ex:alphai}
Let us consider  $f(x)=2b_0(x)+3b_1(x)$, for $q=6$. Since $2^2<q\leq 2^3$ then $h=3$, and $f(x)$ can be rewritten in the form (\ref{eq:1})  as
$$f(x)=b_1(x)+2(b_0(x)+b_1(x))+4\cdot 0,$$
where we now identify $a_0(x)=b_1(x)$, $a_1(x)=b_0(x)+b_1(x)$, and $a_2(x)=0$.
Considering the system $H_{2^3}\Lambda=B,$ where $B=[\zeta^k]^{2^3-1}_{k=0},$ we have that the matrix $\Lambda=[\alpha_i]^{2^3-1}_{i=0}$ is given as $\Lambda=2^{-3}H_{2^3}B,$ i.e.,
\begin{eqnarray*}
 \Lambda=\left(
    \begin{array}{c}
      \alpha_0 \\
      \alpha_1 \\
     \alpha_2 \\
      \alpha_3 \\
      \alpha_4 \\
      \alpha_5 \\
      \alpha_6 \\
      \alpha_7 \\
    \end{array}
  \right)=2^{-3}\left(
            \begin{array}{c}
              \frac{3}{2}+i \frac{\sqrt{3}}{2} \\
              \frac{1}{2}-i\frac{\sqrt{3}}{2} \\
             \frac{3}{2}-i\frac{3\sqrt{3}}{2}\\
              -\frac{3}{2}-i\frac{\sqrt{3}}{2}\\
              -\frac{3}{2}+i\frac{3\sqrt{3}}{2} \\
              \frac{3}{2}+i \frac{\sqrt{3}}{2} \\
              \frac{9}{2}+i \frac{3\sqrt{3}}{2} \\
              \frac{3}{2}-i \frac{3\sqrt{3}}{2} \\
            \end{array}
          \right)
\end{eqnarray*}
In addition, from $\Theta_i(x)=(-1)^{z_{i,0}a_0(x)\oplus z_{i,1}a_1(x)\oplus z_{i,3}\cdot 0}$, $z_i\in \mathbb{Z}^3_2$ we have:
\begin{eqnarray*}
\Theta_0(x)=\Theta_4(x)=1,\;\; \Theta_1(x)=\Theta_5(x)=(-1)^{a_0(x)},\\
 \Theta_2(x)=\Theta_6(x)=(-1)^{a_1(x)},\;\; \Theta_3(x)=\Theta_7(x)=(-1)^{a_0(x)\oplus a_1(x)}.
\end{eqnarray*}
Hence, we have the following calculation:
\begin{eqnarray}\label{zeta1}\nonumber
\zeta^{f(x)}&=&\sum^{2^3-1}_{i=0}\alpha_i\Theta_i(x)=(\alpha_0+\alpha_4)\Theta_0(x)+(\alpha_1+\alpha_5)\Theta_1(x)+(\alpha_2+\alpha_6)\Theta_2(x)+(\alpha_3+\alpha_7)\Theta_2(x)\\ \nonumber
&=&2^{-3}(2\sqrt{3}i\Theta_0(x)+2\Theta_1(x)+6\Theta_2(x)+(-2\sqrt{3}i)\Theta_2(x))\\
&=&2^{-3}(2\sqrt{3}i+2(-1)^{a_0(x)}+6(-1)^{a_1(x)}-2\sqrt{3}i(-1)^{a_0(x)\oplus a_1(x)}).
\end{eqnarray}
Since $a_0(x)=b_1(x)$ and $a_1(x)=b_0(x)+b_1(x)$, for all values of the component functions $b_0(x)$ and $b_1(x)$ we have that $\zeta^{f(x)}$ takes the following values:
\begin{eqnarray*}\label{izraz1}
\zeta^{f(x)}&=&2^{-3}(2\sqrt{3}i+2(-1)^{a_0(x)}+6(-1)^{a_1(x)}-2\sqrt{3}i(-1)^{a_0(x)\oplus a_1(x)}) \\
&=&\left\{\begin{array}{cc}
                                                                                    1, & (b_0(x),b_1(x))=(0,0) \\
                                                                                    \zeta^2=-\frac{1}{2}+i\frac{\sqrt{3}}{2}, & (b_0(x),b_1(x))=(1,0) \\
                                                                                   \zeta^3=-1, & (b_0(x),b_1(x))=(0,1) \\
                                                                                    \zeta^5=\frac{1}{2}-i\frac{\sqrt{3}}{2}, & (b_0(x),b_1(x))=(1,1)
                                                                                  \end{array}
\right..
\end{eqnarray*}
\end{ex}
Hence, from (\ref{X}) we have $\alpha_i=2^{-h}H^{(i)}_{2^h}B,$ for $i=0,\ldots,2^h-1,$ and together with (\ref{p1}) we have that the GWHT is given as
\begin{equation}\label{eq:Wideri}
\mathcal{H}_f(u)=\sum_{x\in \mathbb{Z}^n_2}\zeta^{f(x)}(-1)^{u\cdot x}=\sum_{x\in \mathbb{Z}^n_2}\Big((-1)^{u\cdot x}\sum^{2^h-1}_{i=0}\alpha_i \Theta_i(x)\Big)=\sum^{2^h-1}_{i=0}\alpha_i W_i(u),\hskip 0.5cm \forall u\in\mathbb{Z}^n_2,
\end{equation}
where \begin{equation} \label{eq:W_i}
W_i(u)=\sum_{x\in \mathbb{Z}^n_2}\Theta_i(x)(-1)^{u\cdot x}=\sum_{x\in \mathbb{Z}^n_2}(-1)^{z_{i,0}a_0(x)\oplus \cdots \oplus z_{i,h-1}a_{h-1}(x) \oplus u\cdot x},
\end{equation} i.e., $W_i(u)$ is the WHT of the function $z_{i,0}a_0(x)\oplus \cdots \oplus z_{i,h-1}a_{h-1}(x)$ at point $u\in \mathbb{Z}^n_2$, where $z_i=(z_{i,0},\ldots,z_{i,h-1})\in\mathbb{Z}^{h}_2,$ $i=0,\ldots,2^{h}-1.$ Now we state the main result of this section.

\begin{theo}\label{mainth}
Let $f:\mathbb{Z}^n_2\rightarrow \mathbb{Z}_q$, $2^{h-1}<q\leq 2^h,$ where $f(x)$ is given by (\ref{eq:1}). Let the function $\Theta_i(x)$ be defined by (\ref{theta-sec1.1}), and let $W_i(u)$ denote the WHT of the Boolean function $z_{i,0}a_0(x)\oplus \cdots \oplus z_{i,h-1}a_{h-1}(x)$ at point $u\in \mathbb{Z}^n_2$ as in (\ref{eq:W_i}), for $i=0,\ldots,2^{h}-1.$ Then:
 \begin{enumerate}[1.]
\item $\zeta^{f(x)}$ can be represented as a linear combination of the functions $\Theta_{i}(x),$
\begin{eqnarray*}\label{zeta}
\zeta^{f(x)}=\zeta^{\sum^{h-1}_{i=0}c_ia_i(x)}=\sum^{2^h-1}_{i=0}\alpha_i \Theta_{i}(x),
\end{eqnarray*}
where $\alpha_i$ are given by
\begin{eqnarray*}\label{alpha}
\alpha_i=2^{-h}H^{(i)}_{2^h}B,
\end{eqnarray*}
and the matrix $B$ is given by (\ref{BX}).
\item Consequently, $\mathcal{H}_f(u)$ can be represented as a linear combination of $W_i(u),$ i.e., \begin{equation}\label{eq:Wi_lin}
\mathcal{H}_f(u)=\sum^{2^h-1}_{i=0}\alpha_i W_i(u),\hskip 0.5cm \forall u\in\mathbb{Z}^n_2.
\end{equation}
\end{enumerate}
\end{theo}
For instance, Lemma~\ref{lemma1} is an easy corollary of the above result as illustrated in the following example.
\begin{ex}\label{Bful}
Let $q=8=2^h,$ thus $h=3$, and consider an arbitrary function $f \in \mathcal{GB}^q_n$  given by $f(x)=a_0(x)+2a_1(x)+4a_2(x),$ $f:\mathbb{Z}^n_2\rightarrow \mathbb{Z}_2$. Then, the GWHT of $f$ at some arbitrary point $u \in \mathbb{Z}^n_2$
is given by
\begin{eqnarray*}\label{eq:ex1}
\mathcal{H}_f(u)=\sum_{x\in \mathbb{Z}^n_2}(-1)^{a_{h-1}(x)\oplus u\cdot x}\zeta^{\sum^{h-2}_{i=0}a_i(x)2^i}=\sum_{x\in \mathbb{Z}^n_2}(-1)^{a_{2}(x)\oplus u\cdot x}\zeta^{a_0(x)+2a_1(x)}.
\end{eqnarray*}
Now we would like to represent $\zeta^{a_0(x)+2a_1(x)}$ as a linear combination of functions $\Theta_{0}(x)=1,$ $\Theta_{1}(x)=(-1)^{a_0(x)}$, $\Theta_{2}(x)=(-1)^{a_1(x)}$ and  $\Theta_{3}(x)=(-1)^{a_0(x)+a_1(x)},$ i.e.,
$$\zeta^{a_0(x)+2a_1(x)}=\alpha_0\Theta_{0}(x)+\alpha_1\Theta_{1}(x)+\alpha_2\Theta_{2}(x)+\alpha_3\Theta_{3}(x),$$
where the coefficients $\alpha_i\in\mathbb{C},$ $i=0,1,2,3.$ For such coefficients, all of the following equalities must be true:
\begin{eqnarray*}
\zeta^{a_0(x)+2a_1(x)}=\left\{\begin{array}{cc}
                                1=\alpha_0+\alpha_1+\alpha_2+\alpha_3, & if\;\; (a_0(x'),a_1(x'))=(0,0) \\
                                \zeta^{1}=\alpha_0-\alpha_1+\alpha_2+\alpha_3, & if\;\; (a_0(x'),a_1(x'))=(1,0) \\
                                \zeta^{2}=\alpha_0+\alpha_1-\alpha_2+\alpha_3, & if\;\; (a_0(x'),a_1(x'))=(0,1) \\
                               \zeta^{3}=\alpha_0-\alpha_1-\alpha_2+\alpha_3, & if\;\; (a_0(x'),a_1(x'))=(1,1)
                              \end{array}
\right.,
\end{eqnarray*}
for any input $x'\in \mathbb{Z}^n_2.$
 By Theorem \ref{mainth}, we have $\Lambda=2^{-2}H_{2^2}B$ is given by
\begin{eqnarray*}
\Lambda=\left(
    \begin{array}{c}
      \alpha_0 \\
      \alpha_1 \\
      \alpha_2 \\
      \alpha_3 \\
    \end{array}
  \right)=2^{-2}\left(
                  \begin{array}{cccc}
                    1 & 1 & 1 & 1 \\
                    1 & -1 & 1 & -1 \\
                    1 & 1 & -1 & -1 \\
                    1 & -1 & -1 & 1 \\
                  \end{array}
                \right)\left(
                         \begin{array}{c}
                           1 \\
                           \frac{\sqrt{2}}{2}+i\frac{\sqrt{2}}{2} \\
                           i \\
                           -\frac{1}{\sqrt{2}}+i\frac{1}{\sqrt{2}} \\
                         \end{array}
                       \right)=2^{-2}\left(
                                 \begin{array}{c}
                                   1+(1+\sqrt{2})i \\
                                   1+(1-\sqrt{2})i \\
                                   1+\sqrt{2}-i \\
                                   1-\sqrt{2}-i \\
                                 \end{array}
                               \right).
\end{eqnarray*}
Using $\Lambda$ we obtain Lemma \ref{lemma1}, since for every $u\in \mathbb{Z}^n_2$ we have
\begin{eqnarray*}\label{eq:ex3}
\mathcal{H}_f(u)&=&2^{-\frac{n}{2}}\sum_{x\in \mathbb{Z}^n_2}(-1)^{a_{2}(x)\oplus u\cdot x}\zeta^{a_0(x)+2a_1(x)}=\alpha_0 2^{-\frac{n}{2}}\sum_{x\in \mathbb{Z}^n_2}(-1)^{a_{2}(x)\oplus u\cdot x}+\\
&+&\alpha_1 2^{-\frac{n}{2}}\sum_{x\in \mathbb{Z}^n_2}(-1)^{a_0(x)\oplus a_{2}(x)\oplus u\cdot x}+\alpha_2 2^{-\frac{n}{2}}\sum_{x\in \mathbb{Z}^n_2}(-1)^{a_1(x)\oplus a_{2}(x)\oplus u\cdot x}\\
&+&\alpha_32^{-\frac{n}{2}}\sum_{x\in \mathbb{Z}^n_2}(-1)^{a_0(x)+a_1(x)+a_{2}(x)\oplus u\cdot x}\\
&=&\alpha_0 W_{a_2}(u)+\alpha_1 W_{a_0+a_2}(u)+\alpha_2 W_{a_1+a_2}(u)+\alpha_3 W_{a_0+a_1+a_2}(u).
\end{eqnarray*}
Note that in Lemma \ref{lemma1}, the common factor $2^{-2}$ of the coefficients $\alpha_i$ is moved to the left-hand side by considering $4\mathcal{H}_f(u)$ instead of $\mathcal{H}_f(u)$. Thus, the coefficients $\alpha_i$ above are identical to those in Lemma \ref{lemma1}.
\end{ex}

\section{Sufficient conditions for gbent property}\label{sec:suffic}

In this section, we analyze the conditions under which a generalized function $f\in \mathcal{GB}^q_n$ is gbent, where $n$ may be either even and odd. For even  $q$, we provide sufficient conditions for gbent property in terms of the component functions of $f$. In other words, for this case we give an efficient method for construction of gbent functions using Boolean functions.

Let $f:\mathbb{Z}^n_2\rightarrow \mathbb{Z}_q$ be given in the form (\ref{eq:1}), i.e., $f(x)=\sum^{h-1}_{i=0}a_i(x)2^i,$ and $q$ be even ($2^{h-1}<q\leq 2^h$). For the reasons explained below, we rewrite the function $f(x)$ as
\begin{eqnarray}\label{fform}
f(x)=\frac{q}{2}a(x)+a_0(x)+2a_1(x)+\ldots+2^{p-1}a_{p-1}(x),
\end{eqnarray}
for some $p\leq h-1,$ where $a,a_i\in \mathcal{B}_n$.
We first notice that for $q=2^h$, by simply taking $p=h-1$, the above form is identical to (\ref{eq:1}) after identifying $a(x)=a_{h-1}(x)$.


The importance of the term $\frac{q}{2}a(x)$ is due to the  fact that $\frac{q}{2}$ is the only coefficient from $\mathbb{Z}_q$ for which it holds that $\zeta^{\frac{q}{2}a(x)}=(-1)^{a(x)}.$ This coefficient, which naturally appears when $q=2^h$ as the coefficient of $a_{h-1}(x)$ in (\ref{eq:1}), actually made it possible to express the spectral values of the GWHT of $f$ in terms of certain linear combinations of $W_i$ as given by (\ref{eq:Wi_lin}). This was essentially achieved through an efficient manipulation of the double summation as it was done when deriving (\ref{eq:Wideri}).
 However, we still can not prove that $f$ must contain the term $\frac{q}{2}a(x)$ in this explicit form but assuming this form the derivation of the sufficient conditions when $q \neq 2^h$ becomes much easier.



Hence, using $\zeta^{\frac{q}{2}a(x)}=(-1)^{a(x)}$ and applying Theorem \ref{mainth}-(2) on $\zeta^{a_0(x)+2a_1(x)+\ldots+2^{p-1}a_{p-1}(x)},$ the GWHT at point $u\in\mathbb{Z}^n_2$ is given as:
\begin{eqnarray*}
\mathcal{H}_f(u)=\sum_{x\in \mathbb{Z}^n_2}(-1)^{a(x)\oplus u\cdot x}\sum^{2^{p}-1}_{i=0}\alpha_i \Theta_{i}(x)=\sum^{2^{p}-1}_{i=0}\alpha_i W_i(u),
\end{eqnarray*}
using the same approach as when deriving (\ref{eq:Wideri}).
 Here $W_i(u)$ is WHT at point $u\in\mathbb{Z}^n_2$ of functions $a(x)\oplus z_{i,0}a_0(x)\oplus  \cdots \oplus z_{i,p-1}a_{p-1}(x),$  $z_i=(z_{i,0},\ldots,z_{i,p-1})\in\mathbb{Z}^{p}_2,$ $i=0,\ldots,2^{p}-1.$

 Let us denote the elements of the $i$-th Hadamard row $H^{(i)}_{2^{p}}$ by $h_{i,k}$, $0\leq k, i\leq 2^{p}-1.$ Since the form (\ref{fform}) will impose the system $H_{2^{p}}\Lambda=B,$ where $B=[b_k]^{2^{p}-1}_{k=0}$  and $b_k=\zeta^k$, a  further calculation of GWHT at point $u\in\mathbb{Z}^n_2$ gives:
\begin{eqnarray*}
\mathcal{H}_f(u)&=&\sum^{2^{p}-1}_{i=0}\alpha_i W_i(u)=\sum^{2^{p}-1}_{i=0}\left(2^{-p}\sum^{2^{p}-1}_{k=0}h_{i,k} b_k\right) W_i(u)\\
&=&2^{-p}\sum^{2^{p}-1}_{k=0}\left(\sum^{2^{p}-1}_{i=0}h_{i,k} W_i(u)\right)b_k=2^{-p}\left(\sum^{2^{p}-1}_{k=0}S_k\cos\frac{2\pi k}{q}+i\sum^{2^{p}-1}_{k=0}S_k\sin\frac{2\pi k}{q}\right),
\end{eqnarray*}
where
\begin{eqnarray}\label{Sk}
S_k=\sum^{2^{p}-1}_{i=0}h_{i,k} W_i(u),\hskip 0.5cm k=0,\ldots,2^{p}-1,\; u\in\mathbb{Z}^n_2.
\end{eqnarray}
Defining the column matrices $W=[W_k]^{2^{p}-1}_{k=0}$ and $S=[S_k]^{2^{p}-1}_{k=0}$ we have $S=H_{2^{p}}W$ which in the matrix form is given as,
\begin{eqnarray}\label{Sk2}
W=\left(
                           \begin{array}{c}
                            W_0(u) \\
                           W_1(u) \\
                           \vdots \\
                           W_{2^{p}-1}(u) \\
                           \end{array}
                         \right)_{2^{p}\times 1},\hskip 0.7cm  S=\left(
    \begin{array}{c}
      S_0 \\
      S_1 \\
      \vdots \\
      S_{2^{p}-1} \\
    \end{array}
  \right)_{2^{p}\times 1}=\left(
    \begin{array}{c}
      H^{(0)}_{2^{p}}W \\
      H^{(1)}_{2^{p}}W \\
      \vdots \\
      H^{(2^{p}-1)}_{2^{p}}W \\
    \end{array}
  \right).
\end{eqnarray}
Consequently, we may write $\mathcal{H}_f(u)=2^{-p}(S^T B),$ where $B=[\zeta^k]^{2^{p}-1}_{k=0}$ and $S^T$ is the transpose of $S$. Note that both the matrix $S$ as well as $W$ depend on the input $u,$ and for every $k=0,\ldots,2^p-1,$ we have $S_k=H^{(k)}_{2^p}W,$ since $H_{2^p}$ is symmetric. A well-known property of a Hadamard matrix of the size $2^{p}$ is that  any two distinct rows are orthogonal, thus $\sum_k h_{ik} h_{jk} =0$ for $i \neq j$, and if $i=j$ then $\sum_k h_{ik} h_{jk} =2^{p}.$ The absolute value of $\mathcal{H}_f(u)$ is given as:
\begin{eqnarray}\label{abs}
2^{2p}|\mathcal{H}_f(u)|^2=\left(\sum^{2^{p}-1}_{k=0}S_k\cos\frac{2\pi k}{q}\right)^2+\left(\sum^{2^{p}-1}_{k=0}S_k\sin\frac{2\pi k}{q}\right)^2.
\end{eqnarray}
It is not difficult to see that (\ref{abs}) can be written as
\begin{eqnarray}\label{abs3}
2^{2p}|\mathcal{H}_f(u)|^2= \sum^{2^{p}-1}_{k=0} S^2_k + 2 \sum^{2^{p}-1}_{k=1} \cos \frac{2\pi k}{q}\sum^{2^{p}-1-k}_{\substack{i=0}}S_i S_{i+k}.
\end{eqnarray}
\begin{theo}\label{mainth2}
Let $f:\mathbb{Z}^n_2\rightarrow \mathbb{Z}_q$,
where $f(x)$ is given in the form (\ref{fform}) and $B=[\zeta^k]^{2^{p}-1}_{k=0}.$ Let $W=[W_i(u)]^{2^{p}-1}_{i=0}$ be a column matrix (\ref{Sk2}), where $W_i(u)$ denotes the WHT at point $u\in\mathbb{Z}^n_2$ of the Boolean function $a(x)\oplus z_{i,0}a_0(x)\oplus \cdots \oplus z_{i,p-1}a_{p-1}(x),$ $z_i=(z_{i,0},\ldots,z_{i,p-1})\in\mathbb{Z}^{p}_2,$ $i=0,\ldots,2^{p}-1.$ Then:
\begin{enumerate}[a)]
\item Let $n$ be even and $2^{h-1}<q\leq 2^h$ be even. If all functions $a(x)\oplus z_{i,0}a_0(x)\oplus \ldots \oplus z_{i,p-1}a_{p-1}(x)$ are bent Boolean functions, for every $z_i\in\mathbb{Z}^{p}_2,$ $i=0,\ldots,2^{p}-1,$ and there exists $r\in\{0,1,\ldots,2^{p}-1\}$ so that the transpose of a matrix $W$ defined by (\ref{Sk2}) is equal to $H^{(r)}_{2^{p}}$, i.e., $W^{T}=\pm H^{(r)}_{2^{p}}$ ($\triangle$), then $f(x)$ is gbent.
\item Let $n$ be odd and $q=2^{p+1}=2^h.$ If all functions $a(x)\oplus z_{i,0}a_0(x)\oplus \ldots \oplus z_{i,p-1}a_{p-1}(x)$ are semi-bent Boolean functions, for every $z_i\in\mathbb{Z}^{p}_2,$ $i=0,1,\ldots,2^{p}-1,$ and there exists $r\in\{0,1,\ldots,2^{p}-1\}$ so that $W^{T}=\{\pm\sqrt{2}H^{(r)}_{2^{p-1}},\textbf{0}_{2^{p-1}}\}$ or  $W^{T}=\{\textbf{0}_{2^{p-1}},\pm\sqrt{2}H^{(r)}_{2^{p-1}}\}$  ($\square$)  ($\textbf{0}_{2^{p-1}}$ is the all-zero vector of length $2^{p-1}$), then $f(x)$ is gbent.
\end{enumerate}
\end{theo}
\begin{proof}
$a)$ Let $n$ be even, and let us assume that all functions $a(x)\oplus z_{i,0}a_0(x)\oplus \ldots \oplus z_{i,p-1}a_{p-1}(x)$ are bent Boolean functions, for every $z_i\in\mathbb{Z}^{p}_2,$ $i=0,\ldots,2^{p}-1.$ In addition, let us assume that there exists an integer $r\in\{0,1,\ldots,2^{p}-1\}$ so that $W^{T}=\pm H^{(r)}_{2^{p}}.$ Then the properties of Hadamard matrices in (\ref{Sk2}) imply the following:
\begin{eqnarray*}\label{Sk3}
S=\left(
    \begin{array}{c}
      H^{(0)}_{2^{p}}\cdot W^T \\
      \vdots \\
      H^{(r)}_{2^{p}}\cdot W^T  \\
        \vdots \\
      H^{(2^{p}-1)}_{2^{p}}\cdot W^T \\
    \end{array}
  \right)=\left(
            \begin{array}{c}
              0 \\
              \vdots \\
              H^{(r)}_{2^{p}}\cdot (\pm H^{(r)}_{2^{p}}) \\
              \vdots \\
              0 \\
            \end{array}
          \right)
  =\left(
            \begin{array}{c}
              0 \\
              \vdots \\
              \pm 2^{p} \\
              \vdots \\
              0 \\
            \end{array}
          \right),
\end{eqnarray*}
and for every $i$ and $j$ ($i\neq j$), it holds that $S_iS_j=0.$ Here we regard $H^{(r)}_{2^{p}}$ and $W^T$ as vectors, and using the dot product we may write $S_r=H^{(r)}_{2^{p}}\cdot W^T$. In other words, we use this notation to avoid less precise notation $S_r=H^{(r)}_{2^{p}}W.$ Since in the second sum in (\ref{abs3}) it  is not possible that $S_i=S_{i+k},$ for any $k=1,\ldots,2^{p}-1$ and $i=0,\ldots,2^{p}-1-k$, we get that (\ref{abs3}) is given as
\begin{eqnarray*}
2^{2p}|\mathcal{H}_f(u)|^2=S^2_r=2^{2p},
\end{eqnarray*}
which means that $|\mathcal{H}_f(u)|^2=1,$ i.e., the function $f(x)$ is gbent.\newline
$b)$ Let $n$ be odd and $q=2^{p+1}$. The condition that all functions $a(x)\oplus z_{i,0}a_0(x)\oplus  \ldots \oplus z_{i,p-1}a_{p-1}(x)$  are semi-bent Boolean functions, for every $z_i\in\mathbb{Z}^{p}_2,$ $i=0,1,\ldots,2^{p}-1,$ means that $W_i(u)\in\{0,\pm \sqrt{2}\}.$ First, note that the definition of the Hadamard matrix implies that there are exactly two rows in $H_{2^p}$ whose  first  half of its entries are equal to each other (and second halves contain opposite signs). More precisely, for any $r\in \{0,1,\ldots,2^{p-1}-1\}$ and for rows given as
$$H^{(r)}_{2^p}=\{H^{(r)}_{2^{p-1}}, H^{(r)}_{2^{p-1}}\}\;\; \wedge\;\; H^{(r+2^{p-1})}_{2^p}=\{H^{(r+2^{p-1})}_{2^{p-1}},- H^{(r+2^{p-1})}_{2^{p-1}}\}$$ it holds that $H^{(r)}_{2^{p-1}}=H^{(r+2^{p-1})}_{2^{p-1}}.$ 
Therefore, the condition $W^{T}=\{\pm\sqrt{2}H^{(r)}_{2^{p-1}},\textbf{0}_{2^{p-1}}\}$ or $W^{T}=\{\textbf{0}_{2^{p-1}},\pm\sqrt{2}H^{(r)}_{2^{p-1}}\}$ implies $S_r=\pm 2^{p-1}\sqrt{2}$ and $S_{r+2^{p-1}}=\pm 2^{p-1}\sqrt{2}$, which gives:
\begin{eqnarray*}\label{Skodd}
S=\left(
    \begin{array}{c}
      H^{(0)}_{2^{p}}\cdot W^T \\
      \vdots \\
      H^{(r)}_{2^{p}}\cdot W^T  \\
        \vdots \\
      H^{(r+2^{p-1})}_{2^{p}}\cdot W^T \\
       \vdots \\
      H^{(2^{p}-1)}_{2^{p}}\cdot W^T \\
    \end{array}
  \right)=\left(
            \begin{array}{c}
               0 \\
      \vdots \\
      \pm 2^{p-1}\sqrt{2}  \\
        \vdots \\
     \pm 2^{p-1}\sqrt{2}  \\
       \vdots \\
     0 \\
            \end{array}
          \right).
\end{eqnarray*}
Hence, for every $i\in\{0,\ldots,2^p-1\}\setminus \{r,r+2^{p-1}\}$ we have that $S_i=0.$ It is not difficult to see that all $\sum^{2^p-1-k}_{i=0}S_iS_{i+k}=0$ except for the case when $k=2^{p-1},$ for which we have $\sum^{2^p-1-2^{p-1}}_{i=0}S_iS_{i+2^{p-1}}=S_{r}S_{r+2^{p-1}}=2^{2p-1}.$ However, using $q=2^{p+1}$ in the second sum in (\ref{abs3}), for $k=2^{p-1}$ we have the coefficient $\cos \frac{2\pi k}{q}=\cos \frac{\pi 2^p}{2^{p+1}}=\cos\frac{\pi}{2}=0,$ which means that the whole second sum in (\ref{abs3}) is equal to zero.  Note that $k=2^{p-1}$ does not depend on the integer $r$ in ($\square$), and it is not difficult to see that the only value of $q$ for which $\cos \frac{2\pi k}{q}=0$ is $q=2^{p+1}$ (due to a fact that $q$ is an integer). Consequently, in (\ref{abs3}) we have $$2^{2p}|\mathcal{H}_f(u)|^2=S^2_r+S^2_{r+2^{p-1}}=2^{2p-1}+2^{2p-1}=2^{2p},$$ i.e., $f(x)$ is gbent.\qed
\end{proof}
\begin{rem}
Since $2^{h-1}<q\leq 2^{h},$ it is clear that $p\leq h-1$ in (\ref{fform}). Moreover, the condition $q=2^{p+1}$ in the second statement in Theorem \ref{mainth2} actually means that $q=2^h,$ since it is the only power of $2$ for which it holds $2^{h-1}<q\leq 2^{h}.$ In the case when $n$ and $q$ are even, the gbent functions always exist (consider $f(x)=\frac{q}{2}a(x)$, $a(x)$ any bent Boolean function). The case when $n$ is odd is much more difficult to handle which is also evident through the nonexistence for certain odd $n$ and certain $q$, see e.g.  \cite{Liu}.
\end{rem}
\begin{op}\label{op1}
Prove the converse of Theorem \ref{mainth2}, i.e., prove that the conditions given in Theorem \ref{mainth2} are also necessary.
\end{op}
In what follows we discuss some results which support Open problem \ref{op1}. First, we have the following facts:
\begin{itemize}
\item The converse holds for $q=4$ where the condition ($\triangle$) trivially holds, and the function $f(x)$ is given in the form $f(x)=2a(x)+a_0(x)$ \cite{Tok}, where $n$ is even.
\item When $q=8$ we have Theorem \ref{theo1}, where the conditions $(*)$ and $(**)$ are actually equivalent to conditions ($\triangle$) and ($\square$), respectively (see Section \ref{eqforms}).
\end{itemize}

\subsection{Equivalent forms of conditions ($\triangle$) and ($\square$)}\label{eqforms}

In this section we present two equivalent forms of the condition  ($\triangle$) which are actually imposed by the Hadamard recursion (the same applies on the condition ($\square$)).
Let us discuss the form of the condition ($\triangle$) in Theorem~\ref{mainth2}, where we consider the function $f(x)$ in the form (\ref{fform}). Recall that the condition $(\triangle)$ regards $W^T$ and $H^{(r)}_{2^p}$ as vectors (as mentioned in the proof of Theorem \ref{mainth2}). Hence, for the WHT coefficients $W_i(u)$ at point $u\in \mathbb{Z}^n_2$ defined in Theorem \ref{mainth2} we consider the equality of two vectors given by $$W^T=\{W_0(u),W_1(u),\ldots,W_{2^{p-1}}\}=H^{(r)}_{2^p}.$$
Let $H^{(r)}_{2^{p}}$ ($h\geq 1$) be an arbitrary row of the Hadamard matrix, i.e., $H^{(r)}_{2^{p}}=\{H^{(r)}_{2^{p-1}},\pm H^{(r)}_{2^{p-1}}\},$ where $r\in\{0,1,\ldots,2^p-1\}.$ This implies that for every $t=1,2,\ldots,p$  and $i=0,1,\ldots,2^{t-1}-1$, it holds $h_{r,i}=\pm h_{r,i+2^{t-1}}.$  This further means  that the condition $W^T=\pm H^{(r)}_{2^p}$
is equivalent to a set of equalities
\begin{eqnarray}\label{forma}
W_i(u)=\pm W_{i+2^{t-1}}(u),\;\;  t=1,2,\ldots,p,\;\; i=0,1,\ldots2^{t-1}-1,
\end{eqnarray}
where $u\in \mathbb{Z}^n_2.$ For convenience, to see that indices $t$ and $i$ actually represent the Hadamard recursion, let as consider an example when $p=3$:
\begin{enumerate}[1.]
\item For $t=1$ we have that $i$ takes only the value $0$ and consequently we have $W_i(u)=W_0=\pm W_{i+2^{t-1}}(u)=\pm W_1(u).$ Clearly, for any value of $W_0(u)=\pm 1,$ we have that the vector (row) $\{W_0(u),W_1(u)\}=\{W_0(u),\pm W_0(u)\}$ is always equal to some row of the Sylvester-Hadamard matrix $\pm H_2.$
\item For $t=2$ we have that $i$ takes values $0$ and $1.$ For $i=0$ we have $W_0(u)=\pm W_2(u)$ and $W_1(u)=\pm W_3(u).$ Note that the signs for both equalities are the same. By the previous step and any value $W_0(u)=\pm 1$, we have that the vector $\{W_0(u),W_1(u),W_2(u),W_3(u)\}$ is always equal to some row of the Sylvester-Hadamard matrix $\pm H_{2^2}.$ The same calculation further applies for $t=3=p,$ where $i=0,1,2,$ and we get that $\{W_0(u),\ldots,W_{7}(u)\}$ is always equal to some row of the Sylvester-Hadamard matrix $\pm H_{2^3}.$
\end{enumerate}
It is important to note here that the signs "$\pm$" in every step always depend on the current value of $t.$ For instance, when we previously had $t=1,$ the sign in front of $W_1(u)$ is fixed for all upcoming values of $t> 1$. For $t=2,$ the signs in front of $W_2(u)$ and $W_3(u)$ are also fixed in the same way, etc.

Equivalently, the relation (\ref{forma}) suggests that the condition $W^T=\pm H^{(r)}_{2^p}$ can be written in an equivalent way,
\begin{eqnarray}\label{forma1}
\prod^{2^{t-1}-1}_{i=0}W_i(u)=\prod^{2^{t-1}-1}_{i=0}(\pm W_{i+2^{t-1}}(u)),\;\;  t=1,2,\ldots,p,\;\; i=0,1,\ldots,2^{t-1}-1.
\end{eqnarray}
It is not difficult to see that the condition $(*)$  $W_0(u)W_3(u)=W_1(u)W_2(u)$ in Theorem \ref{theo1} is equivalent to equality (\ref{forma1}) (where $p=3$).

In the case when $n$ is even, the discussion above provides some equivalent forms of the condition ($\triangle$). However, in the case when $n$ is odd we have one additional property on Walsh-Hadamard coefficients $W_i(u)$ in the condition ($\square$). First note that condition $W^{T}=\{\pm\sqrt{2}H^{(r)}_{2^{p}},\textbf{0}_{2^{p}}\}$ or  $W^{T}=\{\textbf{0}_{2^{p}},\pm\sqrt{2}H^{(r)}_{2^{p}}\}$, for some $r\in\{0,1,\ldots,2^p-1\}$, means that we can apply the discussion above on half part of $W^T$, i.e., on $\pm\sqrt{2}H^{(r)}_{2^{p}}.$ Here we mean that signs of half coordinates of  $W^T$ must satisfy the Sylvester-Hadamard recurrence formula. However, for $i=0,1,\ldots2^{p-1}-1$ we have $W_i(u)W_{2^{p}-i-1}(u)=0$ ($t=p$ here), since half coordinates of $W^T$ are zeroes. The equality $W_i(u)W_{2^{p}-i-1}(u)=0,$ for $i=0,1,\ldots2^{p-1}-1,$  $u\in \mathbb{Z}^n_2,$ means that the functions $a(x)\oplus z_{i,0}a_0(x)\oplus z_{i,1}a_1(x)\oplus \ldots \oplus z_{i,p-1}a_{p-1}(x)$ and $a(x)\oplus z_{2^{t}-i-1,0}a_0(x)\oplus z_{2^{t}-i-1,1}a_1(x)\oplus \ldots \oplus z_{2^{t}-i-1,p-1}a_{p-1}(x)$ are disjoint spectra functions \cite{CCA}.


\subsection{Necessary and sufficient conditions for the GMMF class}\label{sec:GMMF}

For any arbitrary positive even integer $q$, an arbitrary gbent function $f:\mathbb{Z}^{2n}_2\rightarrow \mathbb{Z}_q$ that belongs to the GMMF class (for instance see \cite{BentGbent}) is defined as $$f(x,y)=\frac{q}{2}x\cdot \sigma(y)+g(y),$$
where $\sigma$ is a permutation on $\mathbb{Z}^n_2$ and $g:\mathbb{Z}^n_2\rightarrow \mathbb{Z}_q$ an arbitrary generalized function from $\mathcal{GB}^q_n$. We see that here $f(x,y)$ contains the term $\frac{q}{2}a(x),$ where $a(x,y)=x\cdot \sigma(y),$ and therefore only $g(y)$ remains to be described in terms of the component Boolean functions by means of Theorem \ref{mainth2} (due to its connection with $W_i(u)$).

With the following proposition, we prove that all functions from the GMMF class trivially satisfy both conditions in Theorem \ref{mainth2}, and thus they also support Open problem \ref{op1}.
\begin{prop}
Every gbent function from GMMF class satisfies the converse of Theorem \ref{mainth2}.
\end{prop}
\begin{proof}
Let the GMMF function $f:\mathbb{Z}^{2n}_2\rightarrow \mathbb{Z}_q$, for even  $2^{h-1} < q \leq 2^h$, be written in the form  $$f(x,y)=\frac{q}{2}x\cdot \sigma(y)+g(y)=\frac{q}{2}a(x,y)+a_0(y)+2a_1(y)+\ldots+2^{p-1}a_{p-1}(y),$$ where  $p\leq h-1$, $a_i\in \mathcal{B}_{2n}$, $a(x,y)=x\cdot\sigma(y)$, and $g(y)$ is uniquely expressed through $a_i$ as $g(y)=a_0(y)+2a_1(y)+\ldots+2^{p-1}a_{p-1}(y).$ Since $f(x,y)$ is written in the form (\ref{fform}), according to Theorem \ref{mainth2} we have that $W_i(u)$ is the WHT of the function $$a(x,y)\oplus z_{i,0}a_0(y)\oplus \ldots\oplus z_{i,p-1}a_{p-1}(y),$$ where $z_i=(z_{i,0},\ldots,z_{i,p-1})\in \mathbb{Z}^{p}_2,$ $i=0,\ldots,2^{p}-1,$ $u\in \mathbb{Z}^{2n}_2.$
Clearly, for all $i=0,1,\ldots,2^p-1$, it holds that $W_i(u)=\pm 1,$ for every $u\in \mathbb{Z}^n_2$, since all functions above belong to well known Maiorana-McFarland class of bent Boolean functions. This actually proves the first part of converse of Theorem \ref{mainth2}. It only remains to prove that condition ($\triangle$) holds. By relation (\ref{forma}), the condition $(\bigtriangleup)$ is equivalent to the fact that $W_i(u)W_{i+2^{t-1}}(u)$ takes values $\pm 1$  (Section \ref{eqforms}) for all $t=1,2,\ldots,p$ and $i=0,1,\ldots,2^{p-1}-1.$ Let us denote
$$ z^{(i)}(y)=z_{i,0}a_0(y)\oplus \ldots\oplus z_{i,p-1}a_{p-1}(y),$$
for $i=0,1,\ldots,2^{p}-1.$ Now, for every $t=1,2,\ldots,p$, $i=0,1,\ldots,2^{t-1}-1$, and $u=(u_1,u_2)\in \mathbb{Z}^{n}_2\times \mathbb{Z}^{n}_2,$  we have the following calculation:
\begin{eqnarray*}
2^{2n}W_i(u)W_{i+2^{t-1}}(u)&=&\sum_{x,y\in \mathbb{Z}^n_2}(-1)^{a(x,y)\oplus z^{(i)}(y)+u\cdot (x,y)}\sum_{x,y\in \mathbb{Z}^n_2}(-1)^{a(x,y)\oplus z^{(i+2^{t-1})}(y)+u\cdot (x,y)}\\
&=&\left(\sum_{y\in \mathbb{Z}^n_2}(-1)^{ z^{(i)}(y)\oplus u_2 \cdot y}\sum_{x\in \mathbb{Z}^n_2}(-1)^{a(x,y)\oplus u_1\cdot x}\right)\cdot \\
&\cdot&\left(\sum_{y\in \mathbb{Z}^n_2}(-1)^{ z^{(i+2^{t-1})}(y)\oplus u_2\cdot y}\sum_{x\in \mathbb{Z}^n_2}(-1)^{a(x,y)\oplus u_1\cdot x}\right).
\end{eqnarray*}
Since $\sum_{x\in \mathbb{Z}^n_2}(-1)^{a(x,y)\oplus u_1\cdot x}=\sum_{x\in \mathbb{Z}^n_2}(-1)^{x\cdot\sigma(y)\oplus u_1\cdot x}=0,$ unless $\sigma(y)=u_1$ which happens exactly when $y=\sigma^{-1}(u_1).$ In the case $\sigma(y)=u_1,$ then $\sum_{x\in \mathbb{Z}^n_2}(-1)^{x\cdot\sigma(y)\oplus u_1\cdot x}=2^{n}.$ It is not difficult to see that for any $t$, $i$ and $y\in \mathbb{Z}^n_2$, it holds that $z^{(i+2^{t-1})}(y)=z^{(i)}(y)\oplus z^{(2^{t-1})}(y).$ Therefore, we have:
\begin{eqnarray*}
2^{2n}W_i(u)W_{i+2^{t-1}}(u)&=&(2^n(-1)^{z^{(i)}(y)\oplus u_2\cdot y})\cdot(2^n(-1)^{z^{(i+2^{t-1})}(y)\oplus u_2\cdot y})\\
&=&2^{2n}(-1)^{z^{(i)}(y)\oplus z^{(i+2^{t-1})}(y)}=2^{2n}(-1)^{z^{(2^{t-1})}(y)}.
\end{eqnarray*}
where $y=\sigma^{-1}(u_1)$ is fixed, since $u=(u_1,u_2)$ is fixed. Hence, for every $t=1,2,\ldots,p$ and $i=0,1,\ldots,2^{t-1}-1$, we have that $W_i(u)W_{i+2^{t-1}}(u)$ is constant (with value $1$ of $-1$) which corresponds to selected value of $t$, i.e., the condition $(\bigtriangleup)$ is satisfied for every $u\in \mathbb{Z}^{2n}_2$ and arbitrary Boolean functions $a_i\in \mathcal{B}_{2n},$ according to Section \ref{eqforms} and relation (\ref{forma}). Recall that for every (but fixed) value of  $t$ we have that the sign of $W_{i+2^{t-1}}(u)=\pm W_i(u)$ is fixed for all $i=0,1,\ldots,2^{t-1}-1.$ \qed
\end{proof}

\section{Fulfilling  the necessary conditions for gbent property}\label{sec:conditions}

In this section we discuss methods for satisfying the condition ($\triangle$) (or ($\square$)) from Theorem \ref{mainth2}, where we consider $W^T=\pm H^{(r)}_{2^p}$ for some integer $p\geq 1$ and $r\in\{0,1,\ldots,2^p-1\}.$ We discuss certain rather trivial approaches to satisfy these conditions, based on the discussion provided in Section \ref{eqforms}.

In essence, for an arbitrary function $g\in \mathcal{B}_n,$ using the equality $W_g(u)=-W_{g\oplus 1}(u)$ we are able to choose the component functions in Theorem \ref{mainth2} so that the condition ($\triangle$) is satisfied. This actually represents a trivial way to satisfy ($\triangle$), since in that case the equality $W^T=\pm H^{(r)}_{2^h}$ does not depend on $u\in \mathbb{Z}^n_2$.
Another possible method employs a linear  translate of a  function,  which gives a simple relationship between the Walsh spectra of the given function and its translate.  Indeed, if for some fixed $\alpha\in\mathbb{Z}^n_2$ and $g_1,g_2\in \mathcal{B}_n$ we have $g_1(x)=g_2(x\oplus \alpha),$ for all $x\in\mathbb{Z}^n_2$,  then their Walsh spectra are related through $W_{g_1}(u)=(-1)^{u\cdot \alpha}W_{g_2}(u)$, for all $u\in \mathbb{Z}^n_2.$ This equality implies that the condition $W^T=\pm H^{(r)}_{2^p}$ actually depends on $u\in \mathbb{Z}^n_2,$ which means that the integer $r$ may change for different $u\in \mathbb{Z}^n_2.$

\begin{ex}\label{exselect}
In this example  we present a trivial method of satisfying the condition ($\triangle$) using the equality $W_g(u)=-W_{g\oplus 1}(u),$ for any $g\in \mathcal{B}_n$. Let $q=16=2^4$ and $f(x)=a_0(x)+2a_1(x)+2^2a_2(x)+2^{3}a_{3}(x).$ In this case, we have the matrix $W=[W_i(u)]^{2^3-1}_{i=0}$, where $W_i(u)$ is WHT at point $u\in\mathbb{Z}^n_2$ of the function $$a_3(x)\oplus z_{i,0}a_0(x)\oplus z_{i,1}a_1(x)\oplus z_{i,2}a_2(x),$$ $z_i=(z_{i,0},z_{i,1},z_{i,2})\in \mathbb{Z}^3_2.$ Hence, the component functions are chosen in the following way:
\begin{enumerate}[1.]
\item Let $W_0(u)=W_{a_3}(u)$ and $W_1(u)=W_{a_3\oplus a_0}(u)$ be WHTs of two arbitrary bent functions $a_3(x)$ and $a_3(x)\oplus a_0(x)$, i.e., $W_0(u),W_1(u)=\pm 1,$ for any $u\in\mathbb{Z}^n_2$. Assuming that $a_3(x)$ is bent,  we may for instance take $a_0\in \mathcal{A}_n$.  Alternatively, we can select  $a_3(x)$ and $a_0(x)$ to be component functions of some vectorial bent function.
\item Now we must select $a_1(x)$ so that  $a_3(x)\oplus a_1(x)$ and $a_3(x)\oplus a_0(x) \oplus a_1(x)$ are bent,  satisfying additionally $$\{W_0(u),W_1(u)\}=\pm \{W_2(u),W_3(u)\},$$ where
$W_2(u)=W_{a_3\oplus a_1}(u)$ and $W_3(u)=W_{a_3\oplus a_0 \oplus a_1}(u).$ For instance, if we want to have $\{W_0(u),W_1(u)\}=- \{W_2(u),W_3(u)\},$ then we need to choose the function $a_1(x)$ which satisfies $$a_3(x)\oplus a_1(x)=a_3(x) \oplus 1\;\; \wedge\;\; a_3(x)\oplus a_0(x) \oplus a_1(x)=a_3(x)\oplus a_0(x)\oplus 1.$$ Hence, it must the a case that the function $a_1(x)$ is a constant function equal to $1$, i.e., $a_1(x)=1$ for every $x\in\mathbb{Z}^n_2.$ On the other side, selecting $a_1(x)=0,$ for every $x\in\mathbb{Z}^n_2,$ implies $\{W_0(u),W_1(u)\}= \{W_2(u),W_3(u)\}.$
\item Now, the rest of functions are chosen with respect to equality $$\{W_0(u),W_1(u),W_2(u),W_3(u)\}=\pm \{W_4(u),W_5(u),W_6(u),W_7(u)\},$$
where $W_4(u)=W_{a_3\oplus a_2},$ $W_5(u)=W_{a_3\oplus a_0 \oplus a_2}$, $W_6(u)=W_{a_3\oplus a_1 \oplus a_2}$ and $W_7(u)=W_{a_3\oplus a_0\oplus a_1 \oplus a_2}.$ It is not difficult to see that the sign "$+$" imposes $a_2(x)=0$, and the sign $"-"$ imposes $a_2(x)=1$, for every $x\in\mathbb{Z}^n_2.$
\end{enumerate}
Since we started with two arbitrary functions $a_3(x)$ and $a_0(x)$, with first choice "$-$" and second "$+$", it is not difficult to see that all possible values of $W_{a_3}(u)$ and $W_{a_3\oplus a_0}(u)$, due to a previous choice of the component functions, imply that $W^T\in\{\pm H^{(2)}_{2^3},\pm H^{(3)}_{2^3}\}.$
\end{ex}
%

The question whether there exists more non-trivial methods to satisfy the condition $W^T=H^{(r)}_{2^p}$ remains open.
\begin{rem}
In the case when $n$ is odd, satisfying the condition ($\square$) is more complicated, since $W^T$ involves Sylvester-Hadamard signs and disjoint spectra functions.
\end{rem}

\section{Conclusion}\label{sec:concl}
The main contribution of this article is the derivation of a compact and efficient formula for computing the generalized Walsh-Hadamard spectra of generalized Boolean functions and its use for specifying some sufficient conditions for the functions in this class to have gbent property. The main remaining challenge is to address the problem of necessary conditions and to possibly establish the equivalence between the two, at least for some specific instances.

\end{document}